\documentclass[a4paper,10pt]{amsart}

\usepackage{amssymb, enumerate, amssymb}
\usepackage{amstext}
\usepackage{amsmath}
\usepackage{amscd}
\usepackage{latexsym}
\usepackage{amsfonts}
\usepackage{amsthm}
\usepackage{fancybox}
\usepackage[all]{xy}

\newtheorem{thm}{Theorem}[section]
\newtheorem{defn}[thm]{Definition}
\newtheorem{ex}[thm]{Example}
\newtheorem{lem}[thm]{Lemma}
\newtheorem{prop}[thm]{Proposition}

\newtheorem{ass}[thm]{Assertion}

\def\Ker{\mathrm{Ker}}
\def\Im{\mathrm{Im}}
\def\Hom{\mathrm{Hom}}
\def\Q{{\mathbb Q}}
\def\deg{{\rm deg}}

\title{On moduli subspaces of central extensions of rational H-spaces}
\author{Takahito Naito}
\date{}
\address{Department of Mathematical Sciences, Shinshu University, 3-1-1 Asahi, Matsumoto, Nagano 390-8621, Japan}
\email{naito@math.shinshu-u.ac.jp}

\begin{document}
\maketitle

\begin{abstract}
We investigate the moduli sets of central extensions of H-spaces enjoying inversivity, power associativity and Moufang properties. By considering rational H-extensions, it turns out that there is no relationship between the first and the second properties in general.
\end{abstract}

\section{Introduction}
We assume that a space has the homotopy type of a path connected CW-complexes with a nondegenerate base point $*$ and that all maps are based maps.\\
\indent
An {\it H-space} is a space $X$ endowed with  a map $\mu :X\times X \to X$ , called a multiplication, such that both the restrictions $\mu |_{X\times *}$ and $\mu |_{*\times X}$ are homotopic to the identity map of $X$. 
The multiplication naturally induces a binary operation on the homotopy set $[Y,X]$ for any space $Y$. In \cite{Jam1960}, James has proved that $[Y, X]$ is an algebraic loop; that is, it has a two sided unit element and for any elements $x$ and $y\in [Y,X]$, the equations $xa=y$ and $bx=y$ have unique solutions $a,b\in [Y,X]$.\\
\indent
Loop theoretic properties of H-spaces have been considered by several authors, for example, Curjel \cite{Cur1968} and Norman \cite{Nor1963}.
In \cite{AL1990}, Arkowitz and Lupton considered H-space structures on inversive, power associative and Moufang properties. They consider whether there exists an H-space structure which does not satisfy the properties.
Thanks to general theory of algebraic loops, we see that Moufang property implies inversivity and power associativity. However, it is expected that there is no relationship between the latter two properties in general.\\
\indent
In \cite{Kac1995}, Kachi introduced central extensions of H-spaces which are called central H-extensions; see Definition \ref{H-ex}. Roughly speaking, for a given homotopy associative and homotopy commutative H-space $X_{1}$ and an H-space $X_{2}$, a central H-extension of $X_{1}$ by $X_{2}$ is defined to be the product $X_{1}\times X_{2}$ with a twisted multiplication. He also gave a classification theorem for the extensions. In fact, a quotient set of an appropriate  homotopy set classifies the equivalence classes of central H-extensions; see Theorem \ref{thm:Kac}. Such the quotient set is called {\it the moduli set} of H-extensions. Moreover, we refer to the subset of the moduli set corresponding the set of the equivalence classes of H-extensions enjoying a property $P$ via the bijection in the classifying theorem as {\it the moduli subset} of H-extensions associated with the property $P$.\\
\indent
The objective of the paper is to investigate the moduli subsets of central extensions of H-spaces associated with inversive, power associative or Moufang property. If a given H-space is $\Q$-local, then the moduli set of its central H-extensions is endowed with a vector space structure over $\Q$. It turns out that the moduli subset mentioned above inherit the vector space structure. This fact enable us to compare such the moduli subsets as vector space, and also to measure the size of the moduli set with the dimension of the vector space. Thus our main theorem (Theorem \ref{lem:4.4}) deduces the following result.  
\begin{ass}\label{ex1}
Let $S_{{\rm inv}}$, $S_{{\rm p.a}}$ and $S_{{\rm Mo}}$ be the moduli subsets of central H-extensions of the Eilenberg-MacLane space $K(\Q,n)$ by $K(\Q,m)$ associated with inversive, power associative and Moufang properties, respectively.
\begin{enumerate}
\item If $m$ is even, $n=km$ and $k\geq 4$ is even, then $S_{{\rm Mo}}$ is the proper subset of $S_{{\rm p.a}}$ and $S_{{\rm p.a}}$ is the proper subset of $S_{{\rm inv}}:$ $$S_{{\rm Mo}}\subsetneq S_{{\rm p.a}}\subsetneq S_{{\rm inv}}.$$
\item If $m$ is even, $n=km$ and $k\geq 5$ is odd, then $S_{{\rm Mo}}$ is the proper subset of $S_{{\rm p.a}}\cap  S_{{\rm inv}}$. Moreover, $S_{{\rm p.a}}\cap  S_{{\rm inv}}$ is a proper subset of $S_{{\rm p.a}}$ and $S_{{\rm inv}}:$ $$S_{{\rm Mo}}\subsetneq S_{{\rm p.a}} \cap S_{{\rm inv}} \subsetneq S_{{\rm p.a}} \text{ \ and \ } S_{{\rm p.a}} \cap S_{{\rm inv}} \subsetneq S_{{\rm inv}}.$$
\item Otherwise, $S_{{\rm inv}}=S_{{\rm p.a}}=S_{{\rm Mo}}$. 
\end{enumerate}
\end{ass}
\indent
The organization of this paper is as follows. In Section 2, we will recall several fundamental definitions and facts on H-spaces and algebraic loops. 
The classification theorem of central H-extensions are also described. In Section 3, we will present necessary and sufficient conditions for central H-extensions to be inversive, power associative and Moufang. 
The conditions allows us to describe the moduli set of extensions enjoying each of the properties in terms of a homotopy set of maps.
In Section 4, we will deal with rational H-spaces and the dimensions of the moduli spaces mentioned above. Moreover, several examples are presented. Assertion \ref{ex1} is proved at the end of Section 4.

\section{Preliminaries}
We begin by recalling the definition of an algebraic loop.
\begin{defn}{\rm
An} algebraic loop $(Q,\cdot )$ {\rm is a set $Q$ with a map
\begin{eqnarray*}
\cdot : Q\times Q\to Q ; \ (x,y)\longmapsto xy
\end{eqnarray*}
such that the equations $ax=b$ and $ya=b$ have unique solutions $x,y\in Q$ for all $a,b\in Q$. Moreover, there exists an element $e\in Q$ such that $xe=x=xe$ for all $x\in Q$. The element $e$ is called} unit. 
\end{defn}

Algebraic loops with particular multiplications are considered.

\begin{defn}\label{def:2.1.5}
{\rm An algebraic loop $(Q,\cdot)$ is called to be}
\begin{enumerate}
\item inversive {\rm if, to any element $x$ of $Q$, there exists a unique element $x^{-1}$ of $Q$ such that the equalities $x^{-1}x=e=xx^{-1}$ hold,}
\item power associative {\rm if  $x(xx)=(xx)x$ for any $x\in Q$,}
\item Moufang {\rm if $(x(yz))x=(xy)(zx)$ for any $x,y,z\in Q$,}
\item symmetrically associative {\rm if $(xy)x=x(yx)$ for any $x,y\in Q$.}
\end{enumerate} 
\end{defn}

In Section 3, we will focus on symmetrically associative property to examine Moufang property.

\begin{lem}{\rm \cite[Lemma 2A]{Bru1946}}\label{lem:2.3}
If an algebraic loop $Q$ is Moufang, then it is inversive.
\end{lem}

Thus we have implication between properties for multiplication described in Definition \ref{def:2.1.5}.

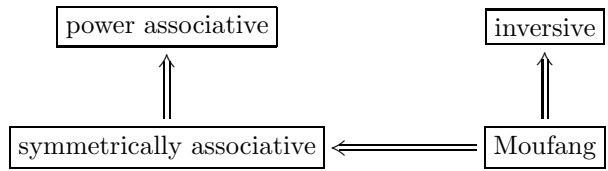
\[\xymatrix{
\fbox{power associative} & &   \fbox{inversive}\\
 \fbox{symmetrically associative} \ar@{=>}[u] & & \fbox{Moufang} \ar@{=>}[ll] \ar@{=>}[u]\\
}\]

We bring the loop theoretic notion to the realm of H-spaces.

\begin{defn}
{\rm Let $(X,\mu)$ be an H-space. A} left inverse $l: X\to X$ {\rm and a} right inverse $r:X\to X$ {\rm of an H-space $(X,\mu)$ are maps such that $\mu (l\times id)\Delta \simeq * \simeq \mu(id\times r)\Delta$.}
\end{defn}

\begin{defn}
{\rm An H-space $(X,\mu)$ is called to be}
\begin{enumerate}
\item inversive {\rm if the left inverse and the right inverse of $(X,\mu )$ are homotopic,}
\item power associative {\rm if  $\mu(\mu \times id)(\Delta \times id)\Delta \simeq \mu(id \times \mu)(\Delta \times id)\Delta$,}
\item Moufang {\rm if $\mu(\mu \times \mu)\theta \simeq \mu(\mu \times id)(id \times \mu \times id)\theta $, where $\theta : X^{3}\to X^{4}$ is defined by $\theta (x,y,z)=(x,y,z,x),$}
\item symmetrically associative {\rm if $\mu(\mu \times id)(id \times t )(\Delta \times id)\simeq \mu(id \times \mu)(id\times t)(\Delta \times id)$, where $t:X^{2}\to X^{2}$ is defined by $t(x,y)=(y,x)$}.
\end{enumerate} 
\end{defn}

The following proposition characterizes the above conditions on H-spaces by homotopy sets.

\begin{prop}
Let $(X,\mu)$ be an H-space. Then $(X,\mu)$ is inversive $($respectively, power associative, Moufang and symmetrically associative$)$ if and only if the homotopy set $[Y,X]$ is inversive $($respectively, power associative, Moufang and symmetrically associative$)$ for any space $Y$.
\end{prop}

\begin{proof}{\rm
Suppose $X$ is inversive. Let $l$ and $r$ be a left inverse and a right inverse of $X$, respectively. Then, for any $[f]\in [Y,X]$, $[lf]$ is the inverse element of $[f]$. Conversely, if $[Y,X]$ is inversive for any space $Y$, then there exists $[\nu ]\in [X,X]$ such that
$
[\nu][id_{X}]=*=[id_{X}][\nu].
$
A similar argument shows the other cases.
}\end{proof}

Here we recall the definition of the central extension of an H-spaces.

\begin{defn}\label{H-ex}{\rm \cite{Kac1995}
Let $(X_{1},\mu_{1})$ be a homotopy associative and homotopy commutative H-space and $(X_{2},\mu_{2})$ an H-space. An H-space $(X,\mu)$ is a} central H-extension of $(X_{1},\mu_{1})$ by $(X_{2},\mu_{2})$ {\rm if there exists a sequence of H-spaces
\begin{eqnarray*}
(X_{1},\mu_{1}) \stackrel{f_{1}}{\longrightarrow } ( X,\mu ) \stackrel{f_{2}}{\longrightarrow }(X_{2},\mu_{2})
\end{eqnarray*}
such that the sequence
\begin{eqnarray*}
e \longrightarrow [Y,X_{1}] \stackrel{f_{1*}}{\longrightarrow } [Y,X] \stackrel{f_{2*}}{\longrightarrow } [Y,X_{2}] \longrightarrow e
\end{eqnarray*}
is exact as algebraic loops for any space $Y$ and the image of $f_{1*}$ is contained in the center of $[Y,X]$, where $f_{i*}$ is the algebraic loop homomorphism induced by $f_{i}$.
}
\end{defn}

In order to describe the classification theorem for central H-extensions presented by Kachi, we recall the definition of H-deviations.
Let $(X,\mu)$ be an H-space and $Y$ a space. For any $[f],[g]\in [Y,X]$, let $D(f,g)$ be a unique element of $[Y,X]$ such that $D(f,g)[g]=[f]$.

\begin{defn}{\rm
Let $(X,\mu)$ and $(Y,\mu')$ be H-spaces, and let $f:X\to Y$ be a map. An} H-deviation of the map f {\rm is an element
\begin{eqnarray*}
HD(f) \in [X\wedge X,Y]
\end{eqnarray*}
such that $q^{*}(HD(f))=D(f\mu , \mu'(f\times f))$, where $q:X\times X \to X\wedge X$ is the quotient map.}
\end{defn}

The existence of H-deviations is insured by the following lemma.

\begin{lem}{\rm \cite[Lemma 1.3.5]{Zab1976}}\label{lem:2.2.2}
Let $(X,\mu)$ be an H-space. Then for any spaces $Y,Z$, 
\begin{eqnarray*}
e\longrightarrow [Y\wedge Z,X] \stackrel{q^{*}}{\longrightarrow  } [Y\times Z,X]\stackrel{i^{*}}{\longrightarrow } [Y\vee Z,X]\longrightarrow e
\end{eqnarray*}
is the short exact sequence as algebraic loops, where $i:Y\vee Z\to Y\times Z$ is the inclusion and $q:Y\times Z\to Y\wedge Z$ is the quotient map.
\end{lem}

Since $i^{*}D(f\mu , \mu'(f\times f))=D(f\mu i , \mu'(f\times f)i)=e$, there exists an element $HD(f) \in [X\wedge X,Y]$ such that $q^{*}(HD(f))=D(f\mu , \mu'(f\times f))$.
If $f\simeq f':X\to X'$, then $HD(f)=HD(f')$. Therefore, the H-deviation map
\begin{eqnarray*}
HD : [X,Y] \to [X\wedge X,Y]
\end{eqnarray*}
is defined by sending a class $[f]$ to the class $HD(f)$. Moreover, if $X$ is a homotopy associative and homotopy commutative H-space, then the H-deviation map is an algebraic loop homomorphism ( \cite[Corollary 2.4]{Kac1995}).

\begin{defn}{\rm \cite{Kac1995}
Two central H-extensions
\begin{eqnarray*}
(X,\mu _{X}) \stackrel{f_{i}}{\longrightarrow } (Z_{i},\mu_{i} ) \stackrel{g_{i}}{\longrightarrow }(Y,\mu_{Y}) \ (i=1,2)
\end{eqnarray*}
is said to be} equivalent {\rm
if there exists an H-map $h:(Z_{1},\mu_{1}) \to (Z_{2},\mu_{2})$ such that the following diagram is homotopy commutative:
$$
\begin{CD}
X @>f_{1}>> Z_{1} @>g_{1}>> Y\\
@V=VV @VhVV @V=VV \\
X @>>f_{2}> Z_{2} @>>g_{2}> Y.\\
\end{CD}
$$
}
\end{defn}

It is readily seen that this relation is an equivalence relation. We denote by $CH(X_{1},\mu_{1};X_{2},\mu_{2})$ the set of equivalence classes of central H-extensions of $(X_{1},\mu_{1})$ by $(X_{2},\mu_{2})$, and by $[(X,\mu),f_{1},f_{2}]$ the equivalence class of the central H-extension
\begin{eqnarray*}
(X_{1},\mu_{1}) \stackrel{f_{1}}{\longrightarrow } (X,\mu ) \stackrel{f_{2}}{\longrightarrow }(X_{2},\mu_{2}).
\end{eqnarray*}

\begin{thm}{\rm \cite[Theorem 4.3]{Kac1995}} \label{thm:Kac}
Let $(X_{1},\mu_{1})$ be a homotopy associative and homotopy commutative H-space, and $(X_{2},\mu_{2})$ an H-space. Define 
\begin{eqnarray*}
\Phi : [X_{2}\wedge X_{2},X_{1}]/\Im HD \longrightarrow  CH(X_{1},\mu_{1};X_{2},\mu_{2})
\end{eqnarray*}
by sending $[\omega ]\in [X_{2}\wedge X_{2},X_{1}]/\Im HD$ to $[(X_{1}\times X_{2},\mu_{\omega}),i_{1},p_{2}]\in CH(X_{1},\mu_{1};X_{2},\mu_{2})$,
where the multiplication $\mu_{\omega}$ of $X_{1}\times X_{2}$ is defined by $\mu_{\omega}((x_{1},x_{2}),(y_{1},y_{2}))=(\mu_{1}(\mu_{1}(x_{1},y_{1}),\omega q(x_{2},y_{2})),\mu_{2}(x_{2},y_{2}))$, $q:X_{2}\times X_{2}\to X_{2}\wedge X_{2}$ is the quotient map. Then $\Phi $ is bijective.
\end{thm}

Thus, for an equivalence class  ${\mathcal E}$ in $CH(X_{1},\mu_{1};X_{2},\mu_{2})$, there exists a map $\omega :X_{2}\wedge X_{2}\to X_{1}$ such that $\Phi [\omega ]=[{\mathcal E}]$, the map $\omega $ is called {\it the classifying map} of the central H-extensions ${\mathcal E}$. Moreover, we refer to the set $[X_{2}\wedge X_{2},X_{1}]/\Im HD$ as {\it the moduli set} of H-extensions of $(X_{1},\mu_{1})$ by $(X_{2},\mu_{2})$.

\section{Moduli subsets of central H-extensions}

We retain the notation and terminology described in the previous section.
Let $(X_{1},\mu_{1})$ be a homotopy associative and homotopy commutative H-space, and $(X_{2},\mu_{2})$ an H-space. Then a central H-extension of $(X_{1},\mu_{1})$ by $(X_{2},\mu_{2})$ is of the form $(X_{1}\times X_{2}, \mu_{\omega})$. Let $i_{j}:X_{j}\to X_{1}\times X_{2}$ and $p_{j}:X_{1}\times X_{2}\to X_{j}$ denote the inclusions and the projections, respectively. Let $\Delta_{j}:X_{j}\to X_{j}\times X_{j}$ and $\Delta:X_{1}\times X_{2}\to (X_{1}\times X_{2})^{2}$ be diagonal maps. We denote by $id_{j}$ and $id$ the identity maps of $X_{j}$ and $X_{1}\times X_{2}$, respectively.

\begin{prop}\label{prop:3.1}
The central H-extension $(X_{1}\times X_{2},\mu_{\omega})$ is inversive if and only if $(X_{2},\mu_{2})$ is inversive and $\omega q(l_{2}\times id_{2})\Delta _{2}\simeq \omega q(id_{2} \times r_{2})\Delta _{2}$, where $l_{2}$ and $r_{2}$ are the left inverse and the right inverse of $(X_{2},\mu_{2})$. 
\end{prop}

Before proving Proposition \ref{prop:3.1}, we prepare lemmas.

\begin{lem}\label{lem:3.2}
Let $(X,\mu)$, $(Y,\mu')$ be H-spaces, and let $f:X\to Y$ be a map. If $f$ is an H-map, then $fl\simeq l'f$, $fr\simeq r'f$, where $l,l'$ are left inverses of $(X,\mu),(Y,\mu')$ and $r,r'$ are right inverses of $(X,\mu),(Y,\mu')$ respectively.
\end{lem}

\begin{proof}{\rm
We see that $[fl][f]=[*]=[l'f][f]$ in $[X,Y]$ since $f$ is an H-map. Hence, we get $[fl]=[l'f]$. Similarly, one obtains the equality $[fr]=[r'f]$.}
\end{proof}

\begin{lem}\label{lem:3.3}
Let $Y,Z$ be spaces, and let $i:Y\to Y\times Z$, $p:Y\times Z\to Y$ be the inclusion and the projection. Then one has
\begin{eqnarray*}
\Ker \{ i^{*}:[Y\times Z,X]\to [Y,X] \} \cap \Im \{ p^{*}:[Y,X]\to [Y\times Z,X] \} = \{ e \}
\end{eqnarray*}
for any H-space $X$. 
\end{lem}

\begin{proof}{\rm
For any $[f]\in \Ker i^{*}\cap \Im p^{*}$, there is a map $g:Y\rightarrow  X$ such that $f\simeq gq$. Since $[f]$ is in $\Ker i^{*}$, it follows that $*\simeq fi\simeq g$ and hence $f\simeq *$.  
}
\end{proof}

\noindent
{\it Proof of Proposition \ref{prop:3.1}.}
Let $l$ and $r$ be a left inverse and a right inverse of $(X_{1}\times X_{2},\mu_{\omega })$, respectively. Note that $l\simeq r$ if and only if $p_{j}l\simeq p_{j}r$ for $j=1,2$. In order to prove Proposition \ref{prop:3.1}, it is enough to show
\begin{enumerate}
\item $p_{1}l \simeq p_{1}r$ if and only if  $\omega q(l_{2}\times id_{2})\Delta _{2}\simeq \omega q(id_{2} \times r_{2})\Delta _{2}$,
\item $p_{2}l\simeq p_{2}r$ if and only if $(X_{2},\mu_{2})$ is inversive.
\end{enumerate}
Since $* \simeq p_{1}\mu_{\omega }(l\times id) \Delta = \mu_{1}(\mu_{1}(p_{1}l\times p_{1})\Delta\times \omega q(p_{2}l\times p_{2})\Delta)\Delta$, it follows that
\begin{eqnarray*}
[*]=([p_{1}l][p_{1}])[\omega q(p_{2}l\times p_{2})\Delta]
\end{eqnarray*}
in $[X_{1}\times X_{2},X_{1}]$. Similarly, we have 
\begin{eqnarray*}
[*]=([p_{1}][p_{1}r])[\omega q(p_{2}\times p_{2}r)\Delta ] 
\end{eqnarray*}
in $[X_{1}\times X_{2},X_{1}].$ Since $[X_{1}\times X_{2},X_{1}]$ is an abelian group, we see that $[p_{1}l]=[p_{1}r]$ if and only if $[\omega q(p_{2}l\times p_{2})\Delta]=[\omega q(p_{2}\times p_{2}r)\Delta].$ Lemma \ref{lem:3.2} allows us to obtain that
\begin{eqnarray*}
[\omega q(p_{2}l\times p_{2})\Delta]=[\omega q(l_{2}\times id_{2} )\Delta_{2}p_{2}]
\end{eqnarray*}
and
\begin{eqnarray*}
[\omega q(p_{2}\times p_{2}r)\Delta]=[\omega q(id_{2}\times r_{2})\Delta_{2}p_{2}].
\end{eqnarray*}
If $[\omega q(l_{2}\times id_{2} )\Delta_{2}p_{2}]=[\omega q(id_{2}\times r_{2})\Delta_{2}p_{2}]$, then 
$[\omega q(l_{2}\times id_{2})\Delta_{2}] =[\omega q(l_{2}\times id_{2} )\Delta_{2}p_{2}i_{2}]=[\omega q(id_{2}\times r_{2})\Delta_{2}p_{2}i_{2}]= [\omega q(id_{2}\times r_{2})\Delta_{2}]$.
Therefore we have the assertion $(1)$. Suppose that $p_{2}l\simeq p_{2}r$. Then $l_{2}\simeq p_{2}li_{2}\simeq p_{2}ri_{2}\simeq r_{2}$ and $(X_{2},\mu_{2})$ is inversive. Conversely, suppose that $(X_{2},\mu_{2})$ is inversive. Then it follows that $p_{2}li_{2}\simeq p_{2}ri_{2}$ and hence
\begin{eqnarray*}
[*]=[p_{2}li_{2}][p_{2}ri_{2}]^{-1}=[p_{2}li_{2}][r_{2}p_{2}ri_{2}]=i_{2}^{*}([p_{2}l][r_{2}p_{2}r]).
\end{eqnarray*}
Therefore we see that $[p_{2}l][r_{2}p_{2}r]$ is in $\Ker i_{2}^{*}.$ On the other hand, $[p_{2}l][r_{2}p_{2}r]=[l_{2}p_{2}][r_{2}r_{2}p_{2}]=p_{2}^{*}([l_{2}][r_{2}r_{2}])\in \Im p_{2}^{*}.$ In view of Lemma \ref{lem:3.3}, we have
\begin{eqnarray*}
[*] = [p_{2}l][r_{2}p_{2}r] =[p_{2}l][p_{2}r]^{-1}.
\end{eqnarray*}
Hence, it follows that $[p_{2}l]=[p_{2}r]$. We have the second assertion.

\begin{prop}\label{prop:3.4}
The central H-extension $(X_{1}\times X_{2},\mu_{\omega})$ is power associative if and only if $(X_{2},\mu_{2})$ is power associative and $\omega q(\mu_{2}\times id_{2})\overline{\Delta}_{2} \simeq \omega q(id_{2} \times \mu_{2})\overline{\Delta}_{2}$, where $\overline{\Delta}_{2}=(\Delta _{2} \times id_{2})\Delta_{2}$.
\end{prop}

\begin{proof}{\rm
As in the proof of proposition \ref{prop:3.1}, it is enough to show the following:
\begin{enumerate}
\item $p_{1}\mu_{\omega }(\mu_{\omega}\times id)\overline{\Delta}\simeq p_{1}\mu_{\omega }(id \times\mu_{\omega})\overline{\Delta}$\\
if and only if  $\omega q(\mu_{2}\times id_{2})\overline{\Delta}_{2} \simeq \omega q(id_{2} \times \mu_{2})\overline{\Delta}_{2}$, where $\overline{\Delta}=(\Delta  \times id)\Delta$.
\item $p_{2}\mu_{\omega }(\mu_{\omega}\times id)\overline{\Delta}\simeq p_{2}\mu_{\omega }(id \times\mu_{\omega})\overline{\Delta}$\\
if and only if $(X_{2},\mu_{2})$ is power associative.
\end{enumerate}
The statement $(1)$ is trivial. Since the map $p_{2}$ is an H-map, the assertion $(2)$ follows.
}
\end{proof}

The same argument as in Proposition \ref{prop:3.4} does work to prove the following propositions. The details are left to the reader.

\begin{prop}\label{prop:3.15}
An H-space $(X_{1}\times X_{2},\mu_{\omega})$ is Moufang if and only if $(X_{2},\mu_{2})$ is Moufang and $[\mu_{1}(\omega q\times \omega q)\theta _{2}][\omega q(\mu_{2}\times \mu_{2})\theta _{2}]=[\mu_{1}(\omega q\times \omega q)(id_{2} \times (\mu_{2} \times id_{2} \times id_{2} )\Delta '_{2} )][\omega q (\mu_{2}(id_{2}\times \mu_{2})\times id_{2})\theta _{2}]\in [X_{2}^{3},X_{1}]$, where $\Delta'_{2}$ is the diagonal map of $X_{2}\times X_{2}$ and $\theta _{2}:X_{2}^{3}\longrightarrow X_{2}^{4}$ is defined by $\theta_{2} (x,y,z)=(x,y,z,x)$.
\end{prop}

\begin{prop}\label{prop:3.17}
$(X_{1}\times X_{2},\mu_{\omega})$ is symmetrically associative if and only if $(X_{2},\mu_{2})$ is symmetrically associative and $\mu_{1}(\omega q\times \omega q)(t \times (id_{2}\times \mu_{2}t)(\Delta _{2}\times id_{2}))\Delta'_{2}\simeq \mu_{1}(\omega q\times \omega q)(id_{2}\times id_{2} \times (\mu_{2}\times id_{2})(id_{2}\times t)(\Delta_{2}\times id_{2}))\Delta'_{2}.$
\end{prop}

Next, we consider the following subsets of the homotopy set $[X_{2}\wedge X_{2},X_{1}]$:
\begin{align*}
G_{{\rm inv}} &= \{ [\omega] \in [X_{2}\wedge X_{2},X_{1}] \ \mid \  \omega q(l_{2}\times id_{2})\Delta _{2}\simeq \omega q(id_{2} \times r_{2})\Delta _{2} \}, \\
G_{{\rm p.a}} &= \{ [\omega] \in [X_{2}\wedge X_{2},X_{1}] \ \mid \  \omega q(\mu_{2}\times id_{2})\overline{\Delta}_{2} \simeq \omega q(id_{2} \times \mu_{2})\overline{\Delta}_{2} \}, \\
G_{{\rm Mo}} &= \{ [\omega] \in [X_{2}\wedge X_{2},X_{1}] \ \mid \ \Gamma  _{{\rm Mo}}(\omega )\simeq \Gamma '_{{\rm 
Mo}}(\omega ) \} \ \text{and} \\
G_{{\rm s.a}} &= \{ [\omega] \in [X_{2}\wedge X_{2},X_{1}] \ \mid \ \mu _{1}(\omega \times \omega )\Gamma  _{{\rm s.a}}\simeq  \mu _{1}(\omega \times \omega )\Gamma ' _{{\rm s.a}}\},
\end{align*}
where
\begin{align*}
&\Gamma  _{{\rm Mo}}(\omega )=\mu_{1}(\mu_{1}(\omega q\times \omega q)\theta _{2}\times \omega q(\mu_{2}\times \mu_{2})\theta _{2})\Delta_{X_{2}^{3}},\\
&\Gamma '_{{\rm Mo}}(\omega )=\mu_{1}(\mu_{1}(\omega q\times \omega q)(id_{2} \times (\mu_{2} \times id_{X_{2}^{2}} )\Delta 
'_{2} )\times \omega q (\mu_{2}(id_{2}\times \mu_{2})\times id_{2})\theta _{2})\Delta_{X_{2}^{3}}, \\
&\Gamma  _{{\rm s.a}}=(q\times q)(t \times (id_{2} \times \mu_{2}t)(\Delta _{2}\times id_{2}))\Delta'_{2}\ \text{and}\\
&\Gamma ' _{{\rm s.a}}=(q\times q)(id_{2}\times id_{2} \times (\mu_{2}\times id_{2})(id_{2}\times t)(\Delta_{2}\times id))\Delta'_{2}.
\end{align*}

\begin{lem}
The sets $G_{{\rm inv}}$, $G_{{\rm p.a}}$, $G_{{\rm Mo}}$ and $G_{{\rm s.a}}$ are subgroups of $[X_{2}\wedge X_{2},X_{1}]$.
\end{lem}

\begin{proof}{\rm
For any $[\omega_{1}],[\omega_{2}]\in G_{{\rm inv}}$, we have
\begin{align*}
[\mu_{1}(\omega_{1}\times \omega_{2})\Delta''_{2} q(l_{2}\times id_{2})\Delta _{2}] 
                      &= [\omega_{1}q(l_{2}\times id_{2})\Delta_{2}][\omega_{2}q(l_{2}\times id_{2})\Delta_{2}] \\
                      &= [\omega_{1}q(id_{2}\times r_{2})\Delta_{2}][\omega_{2}q(id_{2}\times r_{2})\Delta_{2}] \\
                      &= [\mu_{1}(\omega_{1}\times \omega_{2})\Delta''_{2} q(id_{2}\times r_{2})\Delta _{2}]
\end{align*}
where $\Delta''_{2}$ is the diagonal map of $X_{2}\wedge X_{2}$. Hence $[\omega_{1}][\omega_{2}]\in G_{{\rm inv}}$. Let $l_{1}$ be a left inverse of $(X_{1},\mu_{1})$. Then we have
\begin{eqnarray*}
[ l_{1}\omega_{1}q(l_{2}\times id_{2})\Delta _{2}] = [l_{1}\omega_{1}q(id_{2}\times r_{2})\Delta _{2}].
\end{eqnarray*}
Hence we obtain $[\omega_{1}]^{-1}=[l_{1}\omega_{1}]\in G_{{\rm inv}}$. Similarly, we can show that $G_{{\rm p.a}}$, $G_{{\rm Mo}}$ and $G_{{\rm s.a}}$ are subgroups of $[X_{2}\wedge X_{2},X_{1}]$.}
\end{proof}

Let $\Phi : [X_{2}\wedge X_{2},X_{1}]/\Im HD \to CH(X_{1},\mu_{1};X_{2},\mu_{2})$ be the bijection mentioned in Theorem \ref{thm:Kac}.

\begin{lem}\label{lem:4.2}
If $(X_{2},\mu _2)$ is inversive $($respectively, power associative, Moufang and symmetrically associative$)$, then $\Im HD \subset G_{{\rm inv}}$ $($respectively, $\Im HD \subset G_{{\rm p.a}}$, $\Im HD \subset G_{{\rm Mou.}}$ and $\Im HD \subset G_{{\rm s.a}})$.
\end{lem}

\begin{proof}{\rm
For any $[\omega]\in \Im HD$, $\Phi ([\omega])=[(X_{1}\times X_{2},\mu_{1}\times \mu_{2}),i_{1},p_{2}]$. Since $(X_{2},\mu_{2})$ is inversive, $(X_{1}\times X_{2},\mu_{1}\times \mu_{2})$ is also inversive and $\omega q(l_{2}\times id_{2})\Delta _{2}\simeq \omega q(id_{2} \times r_{2})\Delta _{2}$ by Proposition \ref{prop:3.1}. Hence $[\omega]\in G_{{\rm inv}}$. A similar augment shows the other cases.
}
\end{proof}

\begin{thm}\label{prop:4.1.3} The following statements hold.
\ \\[-1.2em]
\begin{enumerate}
\item If $(X_{2},\mu _2)$ is inversive, then $\Phi _{{\rm inv}}:G_{{\rm inv}}/\Im HD\to \{ [(X,\mu),f_{1},f_{2}]\in  CH(X_{1},\mu_{1};X_{2},\mu_{2}) \mid (X,\mu)$ is inversive  $\}$, which is the restricted homomorphism of $\Phi $ to $G_{{\rm inv}}/\Im HD $, is bijective.
\item If $(X_{2},\mu _2)$ is power associative, then $\Phi _{{\rm p.a}}:G_{{\rm p.a}}/\Im HD\to \{ [(X,\mu),f_{1},f_{2}]\in  CH(X_{1},\mu_{1};X_{2},\mu_{2}) \mid (X,\mu)$ is power associative  $\}$, which is the restricted homomorphism of $\Phi $ to $G_{{\rm p.a}}/\Im HD $, is bijective.
\item If $(X_{2},\mu _2)$ is Moufang, then $\Phi _{{\rm Mo}}:G_{{\rm Mo}}/\Im HD\to \{ [(X,\mu),f_{1},f_{2}]\in  CH(X_{1},\mu_{1};X_{2},\mu_{2}) \mid (X,\mu)$ is Moufang $\}$, which is the restricted homomorphism of $\Phi $ to $G_{{\rm Mo}}/\Im HD $, is bijective.
\item If $(X_{2},\mu _2)$ is symmetrically associative, then $\Phi _{{\rm s.a}}:G_{{\rm s.a}}/\Im HD\to \{ [(X,\mu),f_{1},f_{2}]\in  CH(X_{1},\mu_{1};X_{2},\mu_{2}) \mid (X,\mu)$ is symmetrically associative  $\}$, which is the restricted homomorphism of $\Phi $ to $G_{{\rm s.a}}/\Im HD $, is bijective.
\end{enumerate}
\end{thm}

\begin{proof}{\rm
By Proposition \ref{prop:3.1}, we see that the map $\Phi_{{\rm inv}}$ is a well-defined homomorphism and hence it is injective. Let $(X,\mu)$ be an inversive central H-extension of $(X_{1},\mu_{1})$ by $(X_{2},\mu_{2})$. Theorem \ref{thm:Kac} yields that there exists a map $\omega :X_{2}\wedge X_{2}\to X_{1}$ such that $[(X,\mu),f_{1},f_{2}]=[(X_{1}\times X_{2},\mu_{\omega}),i_{1},p_{2}]$. Since $(X,\mu)$ is inversive, $(X_{1}\times X_{2},\mu_{\omega})$ is also inversive. Hence, $\omega q(l_{2}\times id_{2})\Delta _{2}\simeq \omega q(id_{2} \times r_{2})\Delta _{2}$ by Proposition \ref{prop:3.1}. Therefore $\Phi _{{\rm inv}}([\omega]) = [(X,\mu),f_{1},f_{2}]$ and $\Phi_{{\rm inv}}$ is surjective. In the same view, we can show that $\Phi _{{\rm p.a}}$, $\Phi _{{\rm Mo}}$ and $\Phi _{{\rm s.a}}$ are bijective. 
}
\end{proof}

We denote $G_{{\rm inv}}/\Im HD$, $G_{{\rm p.a}}/\Im HD$, $G_{{\rm Mo}}/\Im HD$ and $G_{{\rm s.a}}/\Im HD$ by $S_{{\rm inv}}$, $S_{{\rm p.a}}$, $S_{{\rm Mo}}$ and $S_{{\rm s.a}}$, respectively, and call {\it the moduli subspace} of the H-extensions associated with the corresponding properties.

\section{The moduli spaces of central H-extensions of rational H-spaces}

In this section, we will investigate the moduli set $G/\Im HD$ and its subsets $S_{{\rm inv}}$, $S_{{\rm p.a}}$, $S_{{\rm Mo}}$ and $S_{{\rm s.a}}$ in the rational cases. If $(X,\mu)$ is a $\Q$-local, simply-connected H-space and the homotopy group $\pi _{*}(X)$ are of finite type, then $H^{*}(X;\Q)\cong \Lambda (x_{1},x_{2},\cdots )$ as an algebra, the free commutative algebra with basis $x_{1},x_{2},\cdots $ by Hopf's Theorem \cite[p.286]{Spa1966}.

\begin{thm}{\rm \cite[Proposition 1]{Sch1984}} \label{thm:Sch}
Let $(X,\mu)$ be a $\Q$-local, simply-connected H-space and $\pi _{*}(X)$ are of finite type, let $Y$ be a space. Then the canonical map
\begin{eqnarray*}
[Y,X]\longrightarrow \Hom_{{\rm Alg}}(H^{*}(X;\Q ),H^{*}(Y;\Q )), \ f\longmapsto H^{*}(f)
\end{eqnarray*}
is bijective.
\end{thm}

According to Arkowitz and Lupton \cite{AL1990}, we give $\Hom_{{\rm Alg}}(H^{*}(X;\Q ),H^{*}(Y;\Q ))$ an algebraic loop structure so that the above canonical map is an algebraic loop homomorphism.

\begin{defn}{\rm
Let $M=\Lambda (x_{i};i\in J )$ be a free commutative algebra generated by elements $x_{i}$ for $i\in J$. A homomorphism $\nu :M\to M\otimes M$ is called} a diagonal {\rm if the following diagram is commutative, where $\varepsilon :M\to \Q$ is augmentation:}
\[\xymatrix{
& M \ar[d]^{\nu } \ar[ld]_{\cong } \ar[rd]^{\cong }& \\
\Q \otimes M & M\otimes M \ar[l]^{\varepsilon \otimes id} \ar[r]_{id\otimes \varepsilon } &M\otimes \Q.
}\]
\end{defn}

Let $(X,\mu)$ be an H-space. Then $H^{*}(\mu):H^{*}(X;\Q)\to H^{*}(X;\Q)\otimes H^{*}(X;\Q)$ is a diagonal.

\begin{thm}{\rm \cite[Lemma 3.1]{AL1990}} \label{thm:AL}
Let $M=\Lambda (x_{1},x_{2},\cdots )$ be a free commutative algebra with the diagonal map $\nu :M\to M\otimes M$, and $A$ a graded algebra. We define the product of $\Hom_{{\rm Alg}}(M,A)$ by
\begin{eqnarray*}
\alpha \cdot \beta = m(\alpha \otimes \beta )\nu ,
\end{eqnarray*}
where $m$ is the product of $A$. Then $\Hom_{{\rm Alg}}(M,A)$ is an algebraic loop endowed with the product.
\end{thm}

It follows from the proof of Theorem \ref{thm:AL} elements $\gamma _{1}$ and $\gamma _{2}\in \Hom_{{\rm Alg}}(M,A)$ satisfy the condition $\gamma _{1}\cdot \alpha = \beta$ and $\alpha \cdot \gamma _{2} = \beta $, then
\begin{align*}
&\gamma _{1}(x_{i})=\beta (x_{i})-\alpha (x_{i})-m(\gamma _{1}\otimes \alpha )P(x_{i}) \ \text{and}\\
&\gamma _{2}(x_{i})=\beta (x_{i})-\alpha (x_{i})-m(\alpha \otimes \gamma _{2} )P(x_{i}),
\end{align*}
where $P(x_{i})=\nu (x_{i})-x_{i}\otimes 1 -1\otimes x_{i}.$

\begin{defn}
A left inverse $\lambda$ $(${\rm respectively}, a right inverse $\rho )$ {\rm of an algebraic loop $\Hom_{{\rm Alg}}(M,A)$ are elements of $\Hom_{{\rm Alg}}(M,A)$ such that $\lambda \cdot id = e$ $($respectively, $id \cdot \rho=e )$, where $e$ is the unit of $\Hom_{{\rm Alg}}(M,A)$.
}
\end{defn}

Let $A$ be a graded algebra and $M=\Lambda (x_{i};i\in J )$ a free commutative Hopf algebra for which each $x_{i}$ are primitive $($that is, $\nu (x_{i})=x_{i}\otimes 1+1\otimes x_{i}.)$ We consider the canonical isomorphism
\begin{eqnarray*}
\Psi : \Hom_{{\rm Alg}}(M, A) \stackrel{\cong }{\rightarrow } \Hom_{\Q }(\Q \langle x_{i};j\in J \rangle , A)
\end{eqnarray*}
which is defined by
\begin{eqnarray*}
\Psi (\alpha )(x_{i})=\alpha (x_{i}) \ {\rm for} \ \alpha  \in \Hom_{{\rm Alg}}(M, A),
\end{eqnarray*}
where $\Q \langle x_{i};i\in J \rangle$ denotes the graded $\Q$-vector space with basis $x_{i}$ for $i\in J$. Let $\Hom_{\Q }(\Q \langle x_{i};i\in J \rangle , A)$ denote the set of $\Q$-linear maps from $\Q \langle x_{i};i\in J \rangle$ to $A$. Since $\Hom_{\Q }(\Q \langle x_{i};i\in J \rangle , A)$ is a $\Q$-vector space with respect to the canonical sum and the canonical scalar multiple, we can give the $\Q$-vector space structure to $\Hom_{{\rm Alg}}(M,A)$ via $\Psi $. We regard $\Hom_{\Q }(\Q \langle x_{i};i\in J \rangle , A)$ as an algebraic loop with the canonical sum. Then the isomorphism $\Psi $ is a morphism of algebraic loops. In fact,
\begin{eqnarray*}
\alpha \cdot \beta (x_{i}) = m(\alpha \otimes \beta )\nu (x_{i}) = \alpha (x_{i}) + \beta (x_{i}) = (\alpha +\beta )(x_{i}).
\end{eqnarray*}

Let $(X_{1},\mu_{1})$ be a $\Q$-local, simply-connected, homotopy associative and homotopy commutative H-space and the homotopy group of $X_{1}$ are of finite type, let $(X_{2},\mu_{2})$ be an H-space, and let $H^{*}(X_{1};\Q)=\Lambda (x_{1},x_{2},\cdots )$. Since $(X_{1},\mu_{1})$ is a homotopy associative and homotopy commutative H-space, we see that each $x_{i}$ is primitive; see \cite[Corollary 4.18]{MM1965}.

\begin{lem}\label{lem:3.3.1}
Let $X$ and $Y$ be spaces. Then
\begin{eqnarray*}
H^{n}(X\wedge Y;\Q) \cong  \{ x\otimes y \in H^{n}(X\times Y;\Q) \mid \deg \, x >0 , \deg \, y > 0 \} \ (n\geq 1).
\end{eqnarray*}
\end{lem}

\begin{proof}{\rm
For any $n\geq 1$, we consider the following commutative diagram:
\[\xymatrix{
[X\wedge Y,K(\Q ,n )] \ar[r]^{q^{*}} \ar[d]^{\cong } & [X\times Y, K(\Q ,n)] \ar[r]^{i^{*}} \ar[d]^{\cong } & [X\vee Y,K(\Q ,n)] \ar[d]^{\cong } \\
H^{n}(X\wedge Y;\Q) \ar[r]^{H^{n}(q)} & H^{n}(X\times Y;\Q ) \ar[r]^{H^{n}(i)} & H^{n}(X\vee Y ;\Q ),
}\]
where $K(\Q,n)$ is the Eilenberg-MacLane space. By Lemma \ref{lem:2.2.2}, $H^{n}(q)$ is injective and we have
$$
H^{n}(X\wedge Y;\Q)\cong \Ker H^{n}(i) = \{ x\otimes y \in H^{n}(X\times Y;\Q) \mid \deg \, x >0 , \deg \, y > 0 \}.
$$
}\end{proof}

Define the map
\begin{eqnarray*}
\overline{HD} : \Hom_{\Q}(\Q \langle x_{1} ,x_{2},\cdots \rangle , H^{*}(X_{2} ;\Q ))\rightarrow \Hom_{\Q }(\Q \langle x_{1} ,x_{2},\cdots \rangle , H^{*}(X_{2}\wedge X_{2} ;\Q ))
\end{eqnarray*}
by
\begin{eqnarray*}
\overline{HD}(\alpha )(x_{i})= P_{2}(\alpha (x_{i})) \ {\rm for} \ \alpha \in \Hom_{\Q}(\Q \langle x_{1} ,x_{2},\cdots \rangle , H^{*}(X_{2} ;\Q )),
\end{eqnarray*}
where $P_{2}(x) = H^{*}(\mu_{2})(x) -x\otimes 1 - 1\otimes x \ (x\in H^{*}(X_{2};\Q))$. Then we see that $\overline{HD}$ is a $\Q$-linear map.
\begin{prop}\label{prop:4.4}
The following diagram of algebraic loops is commutative$:$
\[\xymatrix{
\Hom_{\Q}(\Q \langle x_{1},x_{2},\cdots \rangle , H^{*}(X_{2} ;\Q )) \ar[r]^{\hspace{-1.5em}\overline{HD}}  & \Hom_{\Q }(\Q \langle x_{1},x_{2}, \cdots  \rangle , H^{*}(X_{2}\wedge X_{2} ;\Q ))\\
[X_{2},X_{1}] \ar[r]_{HD} \ar[u]^{\Psi H^{*}}_{\cong }   & [X_{2}\wedge X_{2},X_{1}]  \ar[u]_{\Psi H^{*}}^{\cong }.\\
}\]
\end{prop}

\begin{proof}{\rm
By using the canonical isomorphism $\Psi $,
we see that for any $f \in [X_{2},X_{1}]$ there exists a unique element $\overline{D}(f ) \in \Hom_{\Q}(\langle x_{1} ,x_{2},\cdots \rangle , H^{*}(X_{2}\times X_{2} ;\Q )$ such that $\overline{D}(f ) + (H^{*}(f) \otimes H^{*}(f) )\Psi H^{*}(\mu_{1}) =  H^{*}(\mu_{2}) \Psi H^{*}(f) $. Therefore, it follows that $\overline{D}(f)= \Psi H^{*}(D(f\mu_{1}, \mu_{2}(f\times f)))$. Observe that each $x_{i}$ is primitive.
In view of the proof of Theorem \ref{thm:AL}, we have
$$
\overline{D}(f)(x_{i}) = P_{2}(\Psi H^{*}(f)(x_{i})).
$$
Lemma \ref{lem:3.3.1} implies that $\overline{D}(f)(x_{i})=H^{*}(q)\overline{HD}(\Psi H^{*}(f) )(x_{i})$. Hence, by the definition of an H-deviation, we obtain
$$
H^{*}(q)\Psi H^{*}(HD(f)) =\Psi H^{*}(D(f\mu_{1},\mu_{2}(f\times f))) = H^{*}(q)\overline{HD}(\Psi H^{*}(f) ).
$$
Since $H^{*}(q)$ is injective, it follows from Lemma \ref{lem:2.2.2} and Theorem \ref{thm:Sch} that $\Psi H^{*}(HD(f))=\overline{HD}(\Psi H^{*}(f))$.}
\end{proof}

Put $V=\Hom_{\Q }(\Q \langle x_{1} ,x_{2},\cdots \rangle , H^{*}(X_{2}\wedge X_{2} ;\Q ))$ and set
\begin{align*}
V_{{\rm inv}} &= \{ \alpha \in V \ | \ m_{2}(\lambda _{2}\otimes id)H^{*}(q)\alpha = m_{2}(id\otimes \rho _{2})H^{*}(q)\alpha \},\\
V_{{\rm p.a}} &= \{ \alpha \in V \ | \ \overline{m}_{2}(H^{*}(\mu_{2})\otimes id)H^{*}(q)\alpha = \overline{m}_{2}(id\otimes H^{*}(\mu_{2}))H^{*}(q)\alpha \},\\
V_{{\rm Mo}} &= \{ \alpha \in V \ | \ \overline{\Gamma}  _{{\rm Mo}}(\alpha  )=\overline{\Gamma }'  _{{\rm Mo}}(\alpha  ) \},\\
V_{{\rm s.a}} &= \{ \alpha \in V \ | \ H^{*}(\Gamma  _{{\rm s.a}})(\alpha \otimes \alpha) H^{*}(\mu_{1})=H^{*}(\Gamma'  _{{\rm s.a}})(\alpha \otimes \alpha) H^{*}(\mu_{1}) \},
\end{align*}
where $m_{2}$ is the product of $H^{*}(X_{2};\Q)$, $\overline{m}_{2}=m_{2}(m_{2}\otimes id)$, and $\lambda _{2}, \rho _{2}$ are a left inverse and a right inverse of $\Hom_{{\rm Alg}}(H^{*}(X_{2};\Q),H^{*}(X_{2};\Q))$, respectively. Let
$$
\overline{\Gamma}  _{{\rm Mo}}(\alpha  )=H^{*}(\Gamma  _{{\rm Mo}}(\omega )), \ \overline{\Gamma}'  _{{\rm Mo}}(\alpha  )=H^{*}(\Gamma  _{{\rm Mo}}(\omega ))
$$
where $\omega :X_{2}\wedge X_{2}\to X_{1}$ satisfy $H^{*}(\omega)=\alpha $.
 Then, we see that $V_{{\rm inv}}$, $V_{{\rm p.a}}$, $V_{{\rm Mo}}$ and $V_{{\rm s.a}}$ are subspaces of $V$.

\begin{thm}\label{lem:4.4} The following statements hold.
\ \\[-1.2em]
\begin{enumerate}
\item If $(X_{2},\mu _2)$ is inversive, then $\Im \overline{HD} \subset V_{{\rm inv}}$ and the canonical map $S_{{\rm inv}} \to V_{{\rm inv}}/ \Im \overline{HD}$ is an isomorphism of algebraic loops. 
\item If $(X_{2},\mu _2)$ is power associative, then $\Im \overline{HD} \subset V_{{\rm p.a}}$ and the canonical map $S_{{\rm p.a}} \to V_{{\rm p.a}}/ \Im \overline{HD}$ is isomorphism of algebraic loops. 
\item If $(X_{2},\mu _2)$ is Moufang, then $\Im \overline{HD} \subset V_{{\rm Mo}}$ and the canonical map $S_{{\rm Mo}} \to V_{{\rm Mo}}/ \Im \overline{HD}$ is isomorphism of algebraic loops.
\item If $(X_{2},\mu _2)$ is symmetrically associative, then $\Im \overline{HD} \subset V_{{\rm s.a}}$ and the canonical map $S_{{\rm s.a}} \to V_{{\rm s.a}}/ \Im \overline{HD}$ is isomorphism of algebraic loops.
\end{enumerate}
\end{thm}

\begin{proof}{\rm
The assertions follow from Lemma \ref{lem:4.2} and Proposition \ref{prop:4.4}.
}
\end{proof}

\begin{lem}\label{lem:3.3.6}
Let $V^{k}= \Hom_{\Q}(\Q\langle x_{k} \rangle , H^{*}(X_{2}\wedge X_{2};\Q))$ for $k=1,2,\cdots$. Then the $\Q$-linear map
\begin{eqnarray*}
\xi :\displaystyle\bigoplus _{k\geq 1}V^{k} \longrightarrow V ; \ \ \xi (\alpha _{1}, \alpha _{2},\cdots ) (x_{k})=\alpha _{k}(x_{k}) \ (\alpha _{k}\in V^{k})
\end{eqnarray*}
is an isomorphism. Moreover, let
\begin{align*}
& V_{{\rm inv}}^{k} = \{ \alpha \in V^{k} \mid m_{2}(\lambda _{2}\otimes id)H^{*}(q)\alpha = m_{2}(id\otimes \rho _{2})H^{*}(q)\alpha  \} ,\\
& V_{{\rm p.a}}^{k} = \{ \alpha \in V^{k}  \mid \overline{m}_{2}(H^{*}(\mu_{2})\otimes id)H^{*}(q)\alpha = \overline{m}_{2}(id\otimes H^{*}(\mu_{2}))H^{*}(q)\alpha \} ,\\
&V_{{\rm Mo}}^{k} = \{ \alpha \in V^{k} \ | \ \overline{\Gamma}  _{{\rm Mo}}(\alpha  )=\overline{\Gamma }'  _{{\rm Mo}}(\alpha  ) \} \ {\rm and}\\
&V_{{\rm s.a}}^{k} = \{ \alpha \in V^{k} \ | \ H^{*}(\Gamma  _{{\rm s.a}})(\alpha \otimes \alpha) H^{*}(\mu_{1})=H^{*}(\Gamma'  _{{\rm s.a}})(\alpha \otimes \alpha) H^{*}(\mu_{1})\} .
\end{align*}
We define a $\Q$-linear map
\begin{eqnarray*}
\overline{HD}^{ \hspace{0.15em} k} : \Hom_{\Q}(\Q \langle x_{k} \rangle , H^{*}(X_{2} ;\Q ))\rightarrow V^{k}
\end{eqnarray*}
by
\begin{eqnarray*}
\overline{HD}^{\hspace{0.15em}i}(\alpha )(x_{k})= P_{2}(\alpha (x_{k})) , \ \alpha \in \Hom_{\Q}(\Q \langle x_{k} \rangle , H^{*}(X_{2} ;\Q )).
\end{eqnarray*}
Then the restrictions
\begin{align*}
&\displaystyle\bigoplus _{k\geq 1}V_{{\rm inv}}^{k} \longrightarrow V_{{\rm inv}} & &\displaystyle\bigoplus _{k\geq 1}V_{{\rm p.a}}^{k} \longrightarrow V_{{\rm p.a}} &\displaystyle\bigoplus _{k\geq 1}V_{{\rm Mo}}^{k} \longrightarrow V_{{\rm Mo}} \\
&\displaystyle\bigoplus _{k\geq 1}V_{{\rm s.a}}^{k} \longrightarrow V_{{\rm s.a}} & &\displaystyle\bigoplus _{k\geq 1}\Im \overline{HD}^{ \hspace{0.15em} k} \longrightarrow \Im \overline{HD} &
\end{align*}
of $\xi $ to $\displaystyle\bigoplus _{k\geq 1}V_{{\rm inv}}^{k}$, $\displaystyle\bigoplus _{k\geq 1}V_{{\rm p.a}}^{k}$, $\displaystyle\bigoplus _{k\geq 1}V_{{\rm Mo}}^{k}$, $\displaystyle\bigoplus _{k\geq 1}V_{{\rm s.a}}^{k}$ and $\displaystyle\bigoplus _{k\geq 1}\Im \overline{HD}^{ \hspace{0.15em} k}$ are all isomorphisms.
\end{lem}

\begin{proof}
{\rm
The $\Q$-linear map 
\begin{eqnarray*}
\xi ' :V\longrightarrow \displaystyle\bigoplus _{k\geq 1}V^{k}; \ \ \xi '(\alpha )=(\alpha i_{1},\alpha i_{2},\cdots )
\end{eqnarray*}
is an inverse map of $\xi $, where $i_{k}:\Q\langle x_{k} \rangle \longrightarrow \Q \langle x_{1},x_{2},\cdots \rangle$ is the inclusion.
}
\end{proof}

By Lemma \ref{lem:3.3.6}, for the study of $V_{{\rm inv}}/\Im \overline{HD}$, $V_{{\rm p.a}}/\Im \overline{HD}$, $V_{{\rm Mo}}/\Im \overline{HD}$ and $V_{{\rm s.a}}/\Im \overline{HD}$, it is enough to investigate these vector space in the case where $H^{*}(X_{1};\Q)=\Lambda (x)$. In what follows, we give some examples of $V$, $V_{{\rm inv}}$, $V_{{\rm p.a}}$, $V_{{\rm Mo}}$, $V_{{\rm s.a}}$ and $\Im \overline{HD}$.

\begin{ex}\label{ex:4.9}
{\rm
Let $X_{2}$ be a rational H-space such that $H^{*}(X_{2};\Q)=\Lambda (y)$ and that $\deg \, y$ is odd. Then $X_{2}$ is a homotopy associative H-space for any multiplication. Then we have $V=V_{{\rm inv}}=V_{{\rm p.a}}=V_{{\rm Mo}}=V_{{\rm s.a}}$ and $\Im \overline{HD}=\{ 0 \}$.
}
\end{ex}

\begin{ex}\label{ex:3.3.8}{\rm
Let $X_{2}$ be a rational H-space such that $H^{*}(X_{2};\Q)=\Lambda (y)$ and that $\deg \, y$ is even. Then $X_{2}$ is a homotopy associative H-space for any multiplication. Then one has the following:\\
1. If $\deg \, x=2\deg \, y$, then $$V=V_{{\rm inv}}=V_{{\rm p.a}}=V_{{\rm Mo}}=V_{{\rm s.a}}=\Im \overline{HD}\cong  \Q.$$
2. If $\deg \, x=3\deg \, y$, then $$V\cong  \Q^{2} \ \text{and} \ V_{{\rm inv}}=V_{{\rm p.a}}=V_{{\rm Mo}}=V_{{\rm s.a}} = \Im \overline{HD} \cong \{(r_{1},r_{2})\in \Q^{2} \mid  r_{1}=r_{2}\}.$$
3. If $\deg \, x = m\deg \, y$ and $m\geq  4$, then $V\cong \Q ^{m-1},$\\
\hspace{1em} $V_{{\rm inv}}$
$\cong    
\left\{
\begin{array}{l}
V \ (m:{\rm even})\\
\Bigl\{ (r_{1},\cdots ,r_{m-1})\in \Q^{m-1} \mid \displaystyle\sum^{m-1}_{j=1}(-1)^{j+1}r_{j} =0\Bigr\} \ (m:{\rm odd})
\end{array}
\right.
$
\\
\hspace{1em}$V_{{\rm p.a}}\cong \Bigl\{ (r_{1},\cdots ,r_{m-1})\in \Q^{m-1} \mid \displaystyle\sum^{m-1}_{j=1}(2^{j}-2^{m-j})r_{j} =0\Bigr\}$\\
\hspace{1em}$V_{{\rm s.a}}$
$\cong    
\left\{
\begin{array}{l}
\{ (r_{1},\cdots ,r_{m-1})\in \Q^{m-1} \mid r_{l}=r_{m-l}, \  l=1,\cdots , \frac{m-2}{2}\} \ (m:{\rm even})\\
\{ (r_{1},\cdots ,r_{m-1})\in \Q^{m-1} \mid r_{l}=r_{m-l}, \  l=1,\cdots , \frac{m-1}{2}\} \ (m:{\rm odd})
\end{array}
\right.
$\\
\hspace{1em}$\Im \overline{HD}\cong \{ (r_{1},\cdots ,r_{m-1})\in \Q^{m-1} \mid$\\
\hspace{10em} $(m-1)!r_{1}=\cdots =i!(m-i)!r_{i}=\cdots = (m-1)!r_{m-1} \}$
\\
4. Other cases, $V=V_{{\rm inv}}=V_{{\rm p.a}}=V_{{\rm Mo}}=V_{{\rm s.a}}=\Im \overline{HD}=\{ 0 \}$.
}
\end{ex}
\begin{proof}
{\rm 
In the statement 2, we choose the basis $\sigma _{1}$ and $\sigma _{2}$ of V defined by
\begin{eqnarray*}
\sigma _{1}(x)= y\otimes y^{2} \ \text{and} \ \sigma _{2}(x)= y^{2}\otimes y.
\end{eqnarray*}
We define the isomorphism of $\Q$-vector spaces $f:V\to \Q^{2}$ by $f(\sigma _{1})=(1,0)$ and $f(\sigma _{2})=(0,1)$. Then, for any element $\alpha =r_{1}\sigma _{1}+r_{2}\sigma _{2} \ (r_{i}\in \Q)$ of V, we have
\begin{align*}
&m_{2}(\lambda _{2}\otimes id)H^{*}(q)\alpha (x)- m_{2}(id\otimes \rho _{2})H^{*}(q)\alpha (x)= (-2r_{1}+2r_{2})y^{3}, \\
&\overline{m}_{2}(H^{*}(\mu_{2})\otimes id)H^{*}(q)\alpha (x)- \overline{m}_{2}(id\otimes H^{*}(\mu_{2}))H^{*}(q)\alpha (x)=(-2r_{1}+2r_{2})y^{3}.
\end{align*}
The element $\alpha $ is in $\Im \overline{HD}$ if and only if there exists $\beta $ in $\Hom_{\Q}(\Q \langle x \rangle , H^{*}(X_{2} ;\Q ))$ such that $\alpha =\overline{HD}(\beta )$. Put $\beta (x)=ry^{3} \ (r\in \Q)$, then we have $r_{1}=3r=r_{2}$. Hence, $V_{{\rm p.a}}=V_{{\rm Mo}}=V_{{\rm s.a}} = \Im \overline{HD}$. Therefore 2 is shown. We can prove the statements 1,3 and 4 similar to that computations.
}
\end{proof}

\noindent
{\it Proof of Assertion \ref{ex1}.}
If $n=km$ and $k\geq 4$ is even, we have
$$ V_{{\rm Mo}}/\Im HD \subseteq V_{{\rm s.a}}/\Im HD \subsetneq V_{{\rm p.a}}/\Im HD \subsetneq V/\Im HD  =V_{{\rm inv}}/\Im HD$$
by Lemma \ref{lem:3.3.6} and Example \ref{ex:3.3.8}. Thus the statement (1) holds. A similar argument show the other cases.
\hfill\qed

\section{Acknowledgment}
The author is deeply grateful to Katsuhiko Kuribayashi and Ryo Takahashi who provided helpful comments and suggestions.

\end{document}